\documentclass[11pt]{article}
\usepackage[margin=25mm]{geometry}

\usepackage{tabularx}
\usepackage{indentfirst}
\usepackage{setspace}
\usepackage{wrapfig}
\usepackage{bm}

\usepackage{url}
\usepackage{mathrsfs}

\usepackage{graphicx}

\usepackage{tikz-cd}
\usepackage{amscd}
\usepackage{enumerate}
\usepackage{mathtools}
\usepackage{amsmath}
\usepackage{amssymb}
\usepackage{amsthm}
\usepackage{amsfonts}
\usepackage{usebib}
\usepackage{comment}
\usepackage{dsfont}

\newcommand{\N}{\mathbb{N}}
\newcommand{\Z}{\mathbb{Z}}
\newcommand{\Q}{\mathbb{Q}}
\newcommand{\R}{\mathbb{R}}
\newcommand{\C}{\mathbb{C}}

\DeclareMathOperator{\Hom}{Hom}

\DeclareMathOperator{\Ker}{Ker}
\DeclareMathOperator{\Ima}{Im}

\newcommand{\restr}[2]{{
  \left.\kern-\nulldelimiterspace
  #1
  \vphantom{\big|}
  \right|_{#2}
  }}

\theoremstyle{plain}
\newtheorem{Theorem}{Theorem}[section]
\newtheorem{Proposition}[Theorem]{Proposition}
\newtheorem{Lemma}[Theorem]{Lemma}

\theoremstyle{definition}

\newtheorem{Remark}[Theorem]{Remark}

\title{Blanchfield pairings and twisted Blanchfield pairings of torus knots}
\author{Koki Yanagida 
\footnote{
\leftline{E-mail address: {\tt yngdk127@cc.saga-u.ac.jp}}
\leftline{
\,\,\,\,\,\, Faculty of Science and Engineering, Saga University, 1 Honjo,}
\leftline{
\,\,\,\,\,\, Honjo-machi, Saga 840–8502, Japan}}
}
\date{}
\begin{document}
\maketitle

\begin{abstract}\sloppy
We give explicit matrix presentations of the Blanchfield pairing and certain twisted Blanchfield pairings of the $(m,n)$-torus knot $T(m,n)$.
Our method uses a taut identity realizing a genus-two Heegaard splitting of the manifold $X_{T(m,n)}$ obtained from $S^3$ by $0$-surgery along $T(m,n)$.
The taut identity allows us to construct a chain complex of $X_{T(m,n)}$ with few generators.
As a result, we obtain explicit matrix presentations of the Blanchfield pairing of $T(m,n)$.
Moreover, for each Casson-Gordon type metabelian representation and for suitable roots of unity $\xi$ depending on the representation, 
we describe the $(t-\xi)$-primary part of the associated twisted Alexander module and give an explicit description of the restriction of the twisted Blanchfield pairing to this primary summand.
\end{abstract} 

\begin{center}
\normalsize
\baselineskip=11pt
{\bf Keywords} \\
Blanchfield pairing, twisted Blanchfield pairing, torus knots, identities among relations
\end{center}
\begin{center}
\normalsize
\baselineskip=11pt
{\bf Subject Code } \\
Primary 57K10; Secondary 57M27, 20F05.
\end{center}
\large 


\section{Introduction}
For a knot $K$, the Blanchfield pairing is a linking form associated with $K$ \cite{Bla}.
More precisely, it is a nonsingular Hermitian sesquilinear pairing defined on the Alexander module of $K$ with values in a quotient of the Laurent polynomial ring $\Z[t^{\pm 1}]$.
More generally, given a representation of the knot group, one can define a twisted Blanchfield pairing in an analogous way. 
This yields a linking form on the twisted Alexander module of $K$ and the representation \cite{Pow, MP}.
The Blanchfield pairing plays an important role in the study of knot concordance.
Classically, the isomorphism class of the Blanchfield pairing is in one--to--one correspondence with the $S$--equivalence class of a Seifert matrix, and it has been a key tool in the study of concordance for high-dimensional knots \cite{Kea}.
Moreover, several classical invariants can be recovered from the Blanchfield pairing, such as the signatures and the algebraic unknotting number of a knot \cite{BF1, BF2}.
On the other hand, the twisted Blanchfield pairing can reflect four–dimensional properties of knots more strongly than the classical Blanchfield pairing.
In particular, twisted Blanchfield pairings have been crucial in the study of higher–order structures in the knot concordance group \cite{COT, Pow, MP}.  
For example, they can be used to recover a certain difference of Casson–Gordon signatures \cite{BCP2}, and to prove linear independence of certain families of iterated torus knots in the knot concordance group \cite{CKP}.

Therefore, for a given knot $K$, it is useful to have an explicit description of its Blanchfield pairing or twisted Blanchfield pairing.  
In the 1970s, the Blanchfield pairing was first described explicitly in terms of Seifert matrices \cite{Kea, FP}.  
Subsequently, a method to compute both the Blanchfield pairing and the twisted Blanchfield pairing based on a Wirtinger presentation was proposed in \cite{MP}.  
However, in both approaches, the size of the presentation matrices of these pairings grows with the Seifert genus and the number of crossings.
As a consequence, for knots whose Seifert genus and crossing number are both large, explicit computations become difficult using these approaches.
For example, the Seifert genus and the number of crossings of the torus knot $T(m,n)$ are bounded below by $(m-1)(n-1)/2$, and therefore, concise and explicit formulas for the Blanchfield and twisted Blanchfield pairings for $T(m,n)$ have not yet been obtained.
For related computational results, see \cite{Kea, Nos2, BCP3}.
 
The aim of this paper is to give explicit descriptions of the (twisted) Blanchfield pairing of torus knots.
More precisely, we derive closed formulas for the Blanchfield pairing over \(\mathbb{Z}[t^{\pm1}]\), and partial formulas for twisted Blanchfield pairings associated to metabelian representations of Casson--Gordon type.
Our approach does not rely on Seifert matrices and Wirtinger presentations.
Instead, we employ a taut identity realizing a genus--two Heegaard splitting of the $3$-manifold $X_{T(m,n)}$, obtained from $S^3$ by $0$-surgery along $T(m,n)$ (Theorem~\ref{thm:tiden}).
This taut identity provides an explicit cellular chain complex of the universal cover of \(X_{T(m,n)}\) together with a diagonal approximation.
A key point is that the modules and homomorphisms appearing in this chain complex can be chosen with sizes that do not grow with the Seifert genus or the crossing number.
Consequently, we can compute the (twisted) Blanchfield pairings directly from their definition using small matrices, making explicit calculations possible even for torus knots of large genus and crossing number.

We now present the main results of this paper.
We begin with a computational result for the Blanchfield pairing for the torus knot $T(m,n)$.
Let \(m\) and \(n\) be coprime integers greater than \(1\).
Moreover, let \(r\) and \(s\) be the unique integers satisfying
\[mr+ns=1,\qquad -n<r<0<s<m.\]
Let \(T(m,n)\) denote the $(m,n)$-torus knot.
\begin{Theorem}\label{thm:preBl}
Let 
$
\Delta_{T(m,n)}
\coloneqq
 t^{-(m-1)(n-1)/2}\,
 \frac{(1-t)(1-t^{mn})}{(1-t^{m})(1-t^{n})}
$
and
\[
B(m,n)
\coloneqq
 t^{-(m+1)(n+1)/2}\,
 \frac{(1-t^{mr})(1-t^{ns})(t^{m}-t^{n})^{2}}{(1-t^{m})(1-t^{n})}.
\]
Then the Blanchfield pairing associated with $T(m,n)$ is isometric to the following sesquilinear form:
\begin{equation*}
\begin{aligned}
\Z[t^{\pm 1}]/(\Delta_{T(m,n)}) \times \Z[t^{\pm 1}]/(\Delta_{T(m,n)})
&\longrightarrow \Z[t^{\pm 1}]/(\Delta_{T(m,n)}); \\
(f(t), g(t))
&\longmapsto
f(t^{-1}) \, g(t) \, t^{mn} \, B(m,n).
\end{aligned}
\end{equation*}
\end{Theorem}
\noindent
For the proof of Theorem \ref{thm:preBl}, see Section \ref{sec:prBl}.
We next present several results on the twisted Blanchfield pairings of $T(m,n)$.
For \(a\in\mathbb Z_{n}\setminus\{0\}\), let \(\mathcal{O}(a)\) denote the local ring of germs of holomorphic functions at \(e^{2\pi i\,a/n}\), i.e. the ring of convergent power series in $t - e^{2\pi i\,a/n}$.
Let \(\boldsymbol{b}=(b_{1},\dots,b_{m})\in\mathbb Z_{n}^{m}\) satisfy \(b_{1}+\cdots+b_{m}=0\), and let \(\rho(\bm{b}): \pi_1(X_{T(m,n)}) \to \mathrm{GL}_m (\mathcal{O} (a))\) be the corresponding Casson--Gordon type metabelian representation (see Section~\ref{sec:pfp} or \cite{MP, CKP} for the definition).
The following theorem determines the \( (t-e^{2\pi \sqrt{-1} \,a/n}) \)-primary part of the twisted Alexander module associated with the metabelian representation $\rho(\bm{b})$.
\begin{Theorem}\label{thm:pretAl}
Let \(\Theta_{\boldsymbol{b}}(a)\) be the matrix defined in \eqref{eq:Omega}.
For the representation \(\rho(\bm{b})\), the twisted Alexander module $H_{1}\bigl(X_{T(m,n)};\,\mathcal{O}(a)^{m}\bigr)$ is isomorphic to 
$
\mathcal{O}(a)^{m}\big/\Theta_{\boldsymbol{b}}(a)\,\mathcal{O}(a)^{m}.
$
\end{Theorem}
\noindent
For the representation \( \rho (\bm{b}) \), one can define an $\mathcal{O}(a)$-valued Blanchfield pairing on $H_{1}\bigl(X_{T(m,n)};\,\mathcal{O}(a)^{m}\bigr)$.
From the $\mathcal{O}(a)$-valued Blanchfield pairing, we can recover the restriction of the original twisted Blanchfield pairing to the $(t - e^{2 \pi \sqrt{-1} a / n} )$-primary part of $H_{1}\bigl(X_{T(m,n)};\, \C[t^{\pm 1}]^{m}\bigr)$ \cite{CKP}.
Under the assumption $a \in \Z_n \setminus \{0, -b_1, -b_2, \ldots, -b_m\}$, we explicitly determine the $\mathcal{O}(a)$-valued twisted Blanchfield pairing for $T(m,n)$ and $\rho(\bm{b})$ as follows:
\begin{Theorem}\label{thm:SptBL}
Suppose that $a \in \Z_n \setminus \{0, -b_1, -b_2, \ldots, -b_m\}$.
Let $\Theta_{\bm{b}}(a)$ be the matrix defined in \eqref{eq:Omega}, and let $\Psi_{\bm{b}}(a)$ be the matrix defined in \eqref{eq:Psi}.
The $\mathcal{O}(a)$-valued twisted Blanchfield pairing for $T(m,n)$ and \(\rho(\bm{b})\) is isometric to the following sesquilinear form:
\begin{equation*}
\begin{aligned}
\mathcal{O}(a)^m/\Theta_{\bm{b}}(a) \, \mathcal{O}(a)^m
\times
\mathcal{O}(a)^m/\Theta_{\bm{b}}(a) \, \mathcal{O}(a)^m
&\longrightarrow
\mathcal{O}(a) \big/ \bigl( t^{n(m-1)/2} (1 - t^n)^{m-1} \delta_m(t) \bigr); \\
(\bm{f}(t), \bm{g}(t))
&\longmapsto
\bm{f}(t^{-1})^{\top} \, \bigl( t^{n} \Psi_{\bm{b}}(a) \bigr) \, \bm{g}(t) .
\end{aligned}
\end{equation*}
Here, $\delta_m (t)$ is the rational function defined in \eqref{eq:delta}.
\end{Theorem}
\noindent
See Sections~\ref{sec:pretAl} and \ref{sec:SptBL} for the proof of Theorems~\ref{thm:pretAl} and \ref{thm:SptBL}, respectively.

This paper is organized as follows.
In Section \ref{sec:prelim}, we review twisted (co)homology, the Blanchfield pairing, and the twisted Blanchfield pairing.
In Section \ref{sec:cell}, we review identities among relations for oriented closed $3$-manifolds.
In particular, in Section \ref{sec:cell}, we establish the identity of the closed $3$-manifold $X_{T(m,n)}$ obtained by $0$-surgery along the torus knot $T(m,n)$ in $S^3$.
We prove Theorem~\ref{thm:preBl} in Section \ref{sec:prBl}, and we prove Theorems~\ref{thm:pretAl} and \ref{thm:SptBL} in Section \ref{sec:2THM}.

\noindent
\medskip
\textbf{Conventional terminology.}
Throughout this paper, \(K\) denotes a knot in the \(3\)-sphere \(S^{3}\).
Let \(X_{K}\) denote the closed \(3\)-manifold obtained from \(S^{3}\) by \(0\)-surgery on \(K\).
For \(R\in\{\Z, \C\}\), set \(\Lambda_{R}\coloneqq R[t^{\pm1}]\).
Define an involution
\[
(\,\cdot\,)^{\#}:\Lambda_{R}\longrightarrow \Lambda_{R};\qquad
f(t)^{\#}\coloneqq \overline{f(t^{-1})}\quad\text{for }f(t)\in\Lambda_{R},
\]
where \(\overline{(\,\cdot\,)}\) denotes complex conjugation.
For a matrix \(A\), we write \(A^{\top}\) for its transpose.

\subsection*{Acknowledgments}
The author would like to express sincere gratitude to Takefumi Nosaka for his continued support and helpful suggestions. 
The author is also grateful to Anthony Conway and Mark Powell for their valuable comments and suggestions.

\section{Preliminaries}\label{sec:prelim}
In this section, we begin by recalling the twisted (co)homology groups of $X_K$ in Section \ref{sec:twihom}. 
We then recall the definitions of Blanchfield pairings and twisted Blanchfield pairings in Sections~\ref{sec:Bla} and \ref{sec:twiBla}, respectively. 
No new results are proved in this section.

\subsection{Twisted homology and cohomology groups}\label{sec:twihom}
In this subsection, we review twisted (co)homology groups following \cite{Bro}. 
We denote by $\widetilde{X}_K$ the universal covering space of $X_K$.
Let \(M\) denote a left \(\mathbb Z[\pi_{1}(X_{K})]\)-module. 
We endow \(M\) with a right \(\mathbb Z[\pi_{1}(X_{K})]\)-action by
$m\cdot g \coloneqq g^{-1}m$ for $m\in M$ and $g\in \pi_{1}(X_{K})$,
and denote $M$ equipped with this right action by $M^{\mathrm{op}}$.

Since $\pi_1(X_K)$ acts on the chain complex $C_* (\widetilde{X}_K)$ via the deck transformation $\pi_1(X_K) \curvearrowright \widetilde{X}_K$, we may consider $C_* (\widetilde{X}_K)$ as a left $\Z[\pi_1(X_K)]$-module.
Hence we can consider the following chain complex:
\begin{equation*}
C_*(X_K; M) \coloneqq C_* ( \widetilde{X}_K)^{\mathrm{op}} \otimes_{\Z[\pi_1(X_K)]} M.
\end{equation*}
The twisted homology $H_*(X_K;M)$ is defined to be the homology of this chain complex.
Similarly, the twisted cohomology $H^*(X_K;M)$ is the cohomology of the following cochain complex:
\begin{equation*}
C^*(X_K; M) \coloneqq \Hom_{\Z[\pi_1(X_K)]} ( C_* (\widetilde{X}_K), M ).
\end{equation*}
Let $M'$ be a right $\Z[\pi_1(X_K)]$-module.
For a cycle $c \otimes m' \in C_d (X_K; M')$,
the evaluation map $\mathrm{ev}([\omega])$ for a cocycle $\omega$ with cohomology class $[\omega] \in H^d(X_K;M)$ is defined by 
\begin{equation*}
\mathrm{ev}([\omega]) :  H_d ( X_K; M') \longrightarrow M^{\mathrm{op}} \otimes_{\Z[\pi_1(X_K)]} M' ; \qquad
[c \otimes m'] \longmapsto \omega(c) \otimes m'.
\end{equation*}
Moreover, given a diagonal approximation
\begin{equation}\label{eq:diagonal}
D^\sharp : C_3( \widetilde{X}_K ) \longrightarrow C_1( \widetilde{X}_K ) \otimes C_2( \widetilde{X}_K )
\end{equation}
(for the definition, see \cite[p.108]{Bro}),
we can define the cap product by
\begin{equation*}
\begin{aligned}
\frown \ :  H^2(X_K ; M) \otimes H_3(X_K ; \Z) &\longrightarrow H_1(X_K; M);\\
[\omega] \otimes [c \otimes \ell] &\longmapsto [\ell (\mathrm{id}_{C_1( \widetilde{X}_K)} \otimes \omega)( D^\sharp(c) )].
\end{aligned}
\end{equation*}
Here, we regard $\Z$ as the trivial $\Z[\pi_1(X_K)]$-module, and we identify
$\Z \otimes M$ with $M$.
Recall that the Poincar\'e duality map $\mathrm{PD} : H^2( X_K ; M) \xrightarrow{\sim} H_1(X_K; M)$ is defined to be the cap product with the fundamental class of $X_K$.

\subsection{Review of Blanchfield pairings}\label{sec:Bla}
In this subsection, we review Blanchfield pairings.
It is well known (see \cite{Lic}) that $H_1(X_K;\Z) \cong \Z$.
Hence we may consider $H_1(X_K; \Z)$ as the multiplicative group $\langle t \mid \ \rangle$ generated by $t$.
Then $\Lambda_\Z = \Z[t^{\pm 1}]$ becomes a left $\Z[\pi_1(X_K)]$-module via the abelianization map $\pi_1(X_K) \to \langle t \mid \ \rangle$.
We define $\Delta_K \in \Lambda_\Z$ to be a representative of the order of $H_1(X_K; \Lambda_\Z)$ that is invariant under the involution $\#$, i.e.\ $\Delta_K^\# = \Delta_K$.
See \cite[p.5]{FV} for the definition of the order of a module over a UFD. 

Consider the exact sequence
\begin{equation*}
0 \longrightarrow \Lambda_\Z \xrightarrow{\,\,\Delta_K \cdot\,\, } \Lambda_\Z \longrightarrow \Lambda_\Z/(\Delta_K) \longrightarrow 0,
\end{equation*}
which induces a long exact sequence
\begin{equation*}
\begin{aligned}
\cdots &\rightarrow
H^1( X_K; \Lambda_\Z)
\xrightarrow{\Delta_K \cdot}
H^1 (X_K; \Lambda_\Z)
\rightarrow \\
\rightarrow
H^1 (X_K; \Lambda_\Z/(\Delta_K))
&\xrightarrow{\beta}
H^2( X_K; \Lambda_\Z)
\xrightarrow{\Delta_K \cdot}
H^2 (X_K; \Lambda_\Z)
\rightarrow \cdots.
\end{aligned}
\end{equation*}
The connecting homomorphism $\beta$ is usually called the Bockstein homomorphism.
We now show that the Bockstein homomorphism is an isomorphism.
By Poincar\'e duality we have \(H^{2}(X_{K};\Lambda_{\Z})\cong H_{1}(X_{K};\Lambda_{\Z})\).
Since \(\Delta_{K}\) is a representative of the order of \(H_{1}(X_{K};\Lambda_{\Z})\), multiplication by \(\Delta_{K}\) on \(H^{2}(X_{K};\Lambda_{\Z})\) is the zero map.
On the other hand, Poincar\'e duality identifies \(H^{1}(X_{K};\Lambda_{\mathbb Z})\) with \(H_{2}(X_{K};\Lambda_{\mathbb Z})\cong \Lambda_{\mathbb Z}/(1-t)\).
The image of $\Delta_K$ in $\Lambda_\Z/(1-t)$ is a unit since $\Delta_K(1)=\pm 1$ (see \cite{Lic}), and hence
$
\Delta_{K}\,\cdot : H^{1}(X_{K};\Lambda_{\mathbb Z}) \xrightarrow{} H^{1}(X_{K};\Lambda_{\mathbb Z})
$
is an isomorphism.
Therefore, by exactness of the long exact sequence, $\beta$ is an isomorphism.

Using the Poincar\'e duality map $\mathrm{PD}$ and the Bockstein homomorphism $\beta$,
we define a sesquilinear form
\begin{equation}\label{eq:Bl}
\begin{aligned}
\mathrm{Bl}_K: H_1( X_K; \Lambda_\Z) \times H_1( X_K; \Lambda_\Z) &\longrightarrow \Lambda_\Z/(\Delta_K);\\
(x,y) &\longmapsto
\Phi_\Z\Bigl( \mathrm{ev}\bigl((\mathrm{PD} \circ \beta)^{-1}(x)\bigr)(y) \Bigr).
\end{aligned}
\end{equation}
Here, $\Phi_\Z : (\Lambda_\Z/(\Delta_K))^{\mathrm{op}} \otimes \Lambda_\Z \to \Lambda_\Z/(\Delta_K)$ is defined by $\Phi_\Z (f \otimes g ) \coloneqq f^\# g$.
The sesquilinear form $\mathrm{Bl}_K$ is called the Blanchfield pairing. 
It is known that $\mathrm{Bl}_K$ is a nonsingular, sesquilinear, Hermitian form  \cite{Bla}.
Moreover, there are procedures for computing $\mathrm{Bl}_K$ from Seifert matrices or Wirtinger presentations of $K$ \cite{Kea,FP,MP}.
However, it is difficult to explicitly compute the Blanchfield pairing of the torus knot $T(m,n)$ from these procedures since the Seifert genus and the crossing number of $T(m,n)$ are both bounded below by $(m-1)(n-1)/2$.

\begin{Remark}\label{rmk:anoBl}
The Blanchfield pairing can also be described in the following two ways (see, e.g., \cite{Bla, FP, Nos2}).

Let $Q(\Lambda_\Z)$ be the quotient field of $\Lambda_\Z$. 
From the short exact sequence
\begin{equation*}
0 \longrightarrow \Lambda_\Z \xrightarrow{\text{inclusion}} Q(\Lambda_\Z) \longrightarrow Q(\Lambda_\Z)/\Lambda_\Z \longrightarrow 0,
\end{equation*}
we obtain a Bockstein homomorphism $H^1(X_K; Q(\Lambda_\Z)/\Lambda_\Z ) \to H^2(X_K; \Lambda_\Z )$.
This Bockstein homomorphism is also an isomorphism.
Thus, we similarly obtain a sesquilinear form $\widetilde{\mathrm{Bl}}_K : H_1( X_K; \Lambda_\Z) \times H_1( X_K; \Lambda_\Z) \to Q(\Lambda_\Z)/\Lambda_\Z$.
We can check that $\widetilde{\mathrm{Bl}}_K$ corresponds to the composition of $\mathrm{Bl}_K$ with the homomorphism 
\[
\Lambda_\Z/(\Delta_K) \longrightarrow Q(\Lambda_\Z)/\Lambda_\Z;\qquad f \longmapsto f/\Delta_K.
\]

Alternatively, considering the abelianization $\pi_1( S^3 \setminus K) \to \langle t \mid \rangle$, 
we can similarly define a sesquilinear form $H_1(  S^3 \setminus K; \Lambda_\Z) \times H_1(  S^3 \setminus K; \Lambda_\Z) \to \Lambda_\Z/(\Delta_K)$.
From the definitions, one can check that this sesquilinear form coincides with $\mathrm{Bl}_K$ via the isomorphism $H_1(  S^3 \setminus K; \Lambda_\Z) \to H_1(  X_K; \Lambda_\Z)$ induced by the inclusion map $ S^3 \setminus K \subset X_K$.

Consequently, since both of the above sesquilinear forms are isometric to $\mathrm{Bl}_K$, it is enough to consider $\mathrm{Bl}_K$ in \eqref{eq:Bl}.
\end{Remark}

\subsection{Review of twisted Blanchfield pairings}\label{sec:twiBla}
We recall the definition of the twisted Blanchfield pairing in this subsection.
Fix an integer $\ell \ge 1$.
Let $\rho \colon \pi_1(X_K) \to \mathrm{GL}_\ell (\Lambda_\C)$ be a nontrivial representation satisfying the following three conditions:
\begin{enumerate}
\item[(I)] $\rho(g^{-1}) = \rho(g)^{\# \, \top}$ for any $g \in \pi_1(X_K)$.
\item[(II)] $H_i(X_K; Q(\Lambda_\C)^\ell) = 0$ for every $i \in \Z$, where $Q(\Lambda_\C)$ denotes the field of fractions of $\Lambda_\C$.
\item[(III)] $H^1(X_K; \Lambda_\C^\ell)=0$.
\end{enumerate}
From the condition~(II), $H_1(X_K; \Lambda_\C^\ell)$ is a torsion $\Lambda_\C$-module.
Moreover, condition~(I) implies that its order is invariant under the involution $\#$ (see \cite{FV}).
Thus we can choose an element $\Delta_K^\rho \in \Lambda_\C$ as a representative of the order of $H_1(X_K; \Lambda_\C^\ell)$ such that $(\Delta_K^\rho)^\# = \Delta_K^\rho$.

Consider the short exact sequence
\begin{equation*}
0 \longrightarrow \Lambda_\C^\ell \xrightarrow{\,\,\Delta^\rho_K \cdot\,\, } \Lambda_\C^\ell \longrightarrow (\Lambda_\C/(\Delta^\rho_K))^\ell \longrightarrow 0
\end{equation*}
which induces a long exact sequence
\begin{equation*}
\begin{aligned}
\cdots &\rightarrow
H^1( X_K; \Lambda_\C^\ell)
\xrightarrow{\Delta^\rho_K \cdot}
H^1 (X_K; \Lambda_\C^\ell)
\rightarrow \\
\rightarrow
H^1 (X_K; (\Lambda_\C/(\Delta^\rho_K))^\ell )
&\xrightarrow{\beta}
H^2( X_K; \Lambda_\C^\ell)
\xrightarrow{\Delta^\rho_K \cdot}
H^2 (X_K; \Lambda_\C^\ell)
\rightarrow \cdots .
\end{aligned}
\end{equation*}
From the assumptions on $\rho$ and the definition of $\Delta^\rho_K$, the Bockstein homomorphism $\beta$ is an isomorphism.

We define the twisted Blanchfield pairing $\mathrm{Bl}_K^{\rho}$ by 
\begin{equation}\label{eq:twiBL}
\begin{aligned}
\mathrm{Bl}_K^{\rho}: H_1( X_K; \Lambda_\C^\ell) \times H_1( X_K; \Lambda_\C^\ell) &\longrightarrow \Lambda_\C/(\Delta^\rho_K);\\
(x,y) &\longmapsto
\Phi_\C \bigl( \mathrm{ev}\bigl((\mathrm{PD} \circ \beta)^{-1}(x)\bigr)(y) \bigr).
\end{aligned}
\end{equation}
Here, $\Phi_\C : ((\Lambda_\C/(\Delta^\rho_K) ) ^\ell)^{\mathrm{op}} \otimes \Lambda_\C^\ell \to \Lambda_\C/(\Delta^\rho_K)$ is defined by $\Phi_\C ( \bm{f} \otimes \bm{g} ) \coloneqq \bm{f}^{\# \top} \bm{g}$ for $\bm{f} \in (\Lambda_\C/(\Delta^\rho_K) )^\ell$ and $\bm{g} \in \Lambda_\C^\ell$.
The twisted Blanchfield pairing $\mathrm{Bl}^\rho_K$ is a nonsingular, sesquilinear, Hermitian form with respect to $\#$ (see \cite{Pow} and \cite[Proposition~5.3]{MP}).

Since the above discussion remains valid if we replace the coefficients $\Lambda_\C$ by the local ring $\mathcal{O}_\xi$ of germs of holomorphic functions at a root of unity $\xi$, we obtain the following sesquilinear form:
\begin{equation}\label{eq:OBl}
H_1( X_K; \mathcal{O}_\xi^\ell) \times H_1( X_K; \mathcal{O}_\xi^\ell) \longrightarrow \mathcal{O}_\xi/(\Delta^\rho_K).
\end{equation}
This sesquilinear form recovers $\mathrm{Bl}^\rho_K$ restricted to the $(t - \xi)$-primary part of $H_1(X_K; \Lambda_\C^\ell)$ (see \cite[Proposition~3.4]{BCP1}).
In this paper, with a slight abuse of notation, we also denote the sesquilinear form \eqref{eq:OBl} by $\mathrm{Bl}_K^\rho$.
\begin{Remark}
Let $Q(\Lambda_\C)$ denote the quotient field of $\Lambda_\C$.
In general, the twisted Blanchfield pairing is defined as the sesquilinear form
$\widetilde{\mathrm{Bl}}_K^{\rho}: 
H_1( X_K; \Lambda_\C^\ell) \times H_1( X_K; \Lambda_\C^\ell) \longrightarrow Q(\Lambda_\C)/\Lambda_\C$ similarly to Remark~\ref{rmk:anoBl} (for detailed definition, see \cite[Section~4.4]{MP} and \cite[Section 3.1]{BCP3}).
One can check that $\widetilde{\mathrm{Bl}}_K^{\rho}$ is isometric to the pairing obtained by composing $\mathrm{Bl}_K^{\rho}$ with the homomorphism
\[
\Lambda_\C/(\Delta_K^\rho) \longrightarrow Q(\Lambda_\C)/\Lambda_\C;\qquad f \longmapsto f/\Delta_K^\rho.
\]
In this paper, we adopt \eqref{eq:twiBL} as the definition of the twisted Blanchfield pairing for computational purposes.
\end{Remark}
\begin{Remark}\label{rmk:bock}
In the definition of the (twisted) Blanchfield pairing, we use the Bockstein homomorphism $\beta$, which can be realized at the cochain level as follows.
Let $R \in \{\Z,\C\}$, and let $M$ be a free $\Lambda_R$-module.
For $\Delta \in \Lambda_R$, consider the short exact sequence
\[
0 \xrightarrow{} M \xrightarrow{\Delta\cdot} M \xrightarrow{\mathrm{proj}} M/ \Delta M \xrightarrow{} 0.
\]
This short exact sequence induces a long exact sequence
\begin{equation*}
\begin{aligned}
\cdots &\rightarrow
H^1( X_K; M)
\xrightarrow{\Delta \cdot} 
H^1 (X_K; M)
\rightarrow \\
\rightarrow
H^1 (X_K; M/ \Delta M)
&\xrightarrow{\beta}
H^2( X_K; M)
\xrightarrow{\Delta \cdot}
H^2 (X_K; M)
\rightarrow \cdots .
\end{aligned}
\end{equation*}
The connecting homomorphism $\beta$ in this long exact sequence is the Bockstein homomorphism.
Suppose that we have a splitting $\mathfrak{s}\colon M/\Delta M \to M$ of $\mathrm{proj}: M \twoheadrightarrow M/\Delta M$.
Let \(\mathfrak{s}^*:C^{1}(X_{K}; M/\Delta M)\to C^{1}(X_{K}; M)  \) denote the map induced by \(\mathfrak{s}\), and let \(\partial^{2}: C^{1}(X_{K}; M)\to C^{2}(X_{K}; M) \) be the coboundary homomorphism.
Consider the short exact sequence of complexes
\[
0 \longrightarrow C^*(X_K;M) \xrightarrow{\ \Delta \cdot\ } C^*(X_K;M) \xrightarrow{\ \mathrm{proj}_*\ } C^*(X_K;M/\Delta M) \longrightarrow 0,
\]
which can be written in low degrees as
\[\begin{tikzcd}[ampersand replacement=\&]
	0 \& {C^2 ( X_K ; M )} \& {C^2 ( X_K ; M )} \& {C^2 ( X_K ; M/\Delta M )} \& 0 \\
	0 \& {C^1 ( X_K ; M )} \& {C^1 ( X_K ; M )} \& {C^1 ( X_K ; M/\Delta M )} \& 0.
	\arrow[from=1-1, to=1-2]
	\arrow["{\Delta \cdot}", from=1-2, to=1-3]
	\arrow[from=1-3, to=1-4]
	\arrow[from=1-4, to=1-5]
	\arrow[from=2-1, to=2-2]
	\arrow[from=2-2, to=1-2]
	\arrow[from=2-2, to=2-3]
	\arrow["{\partial^2}", from=2-3, to=1-3]
	\arrow["{\mathrm{proj}_*}", from=2-3, to=2-4]
	\arrow[from=2-4, to=1-4]
	\arrow[from=2-4, to=2-5]
\end{tikzcd}\]
A diagram chase shows that the image of $\partial^2 \circ \mathfrak{s}^*$ is contained in
$\Delta \cdot C^2(X_K;M)$.
Since multiplication by $\Delta$ on $C^2(X_K;M)$ is injective, we obtain
\[
\dfrac{1}{\Delta}\;\partial^{2}\circ\mathfrak{s}^* \colon
C^{1}(X_{K}; M/\Delta M) \longrightarrow C^{2}(X_{K}; M),
\]
which induces the Bockstein homomorphism on cohomology.
\end{Remark}

\section{Recovering chain complexes of $X_K$ from taut identities}\label{sec:cell}
To compute the Blanchfield pairing and the twisted Blanchfield pairing, we give an explicit description of a chain complex $C_* (\widetilde{X}_K)$ together with a diagonal approximation map $D^\sharp$.
In this section, we first recall how to compute the chain complex via Fox derivatives and taut identities from \cite{Sie,Tro,Nos1}.
Then we give an explicit taut identity corresponding to $X_{T(m,n)}$ (see Theorem \ref{thm:tiden}).

\subsection{Taut identities}\label{sec:fox}
In this subsection, we briefly review taut identities.
No new results are proved here.

\noindent
Suppose that $X_K$ has a genus-$g$ Heegaard splitting, which gives rise to a group presentation 
\begin{equation}\label{eq:genpre}
\pi_1(X_K) = \langle x_i \ (i = 1, \ldots, g ) \mid \gamma_i \ (i = 1, \ldots, g) \rangle.
\end{equation}
Here $x_i$ and $\gamma_i$ correspond to the $1$-handles and $2$-handles of the Heegaard splitting, respectively. 
Let
\[
F_g \coloneqq \langle x_i \ (i = 1, \ldots, g ) \mid \ \rangle
\quad\text{and}\quad
R_g \coloneqq \langle \widetilde{\gamma}_i \ (i = 1, \ldots, g) \mid \ \rangle
\]
be free groups of rank $g$.
In addition, for $i=1,\ldots,g$, we define
$[\,\cdot\,] : F_g \longrightarrow \pi_1(X_K)$ and $\psi : R_g*F_g \longrightarrow F_g$
by 
\[
[x_i] \coloneqq x_i, \qquad \psi(x_i) \coloneqq x_i, \qquad \psi(\widetilde{\gamma}_i) \coloneqq \gamma_i.
\]
An element $\sigma \in R_g*F_g$ is called an identity among relations if $\sigma \in \ker \psi$ and $\sigma$ can be written as
$
\sigma = \prod_{i=1}^{\ell} \omega_i \widetilde{\gamma}_{j_i}^{\varepsilon_i} \omega_i^{-1}
$
for some $\ell\in\N$, $\varepsilon_i \in \{ \pm 1 \}$, $j_i \in \{ 1, \ldots, g\}$ and $\omega_i \in F_g$.

For the presentation \eqref{eq:genpre} and an identity $\sigma \in R_g * F_g$, consider the complex
\begin{equation}\label{eq:cellofuniv}
\Z[\pi_1(X_K)] 
\xrightarrow{\partial_3}
\Z[\pi_1(X_K)]^g 
\xrightarrow{\partial_2}
\Z[\pi_1(X_K)]^g
\xrightarrow{\partial_1}
\Z[\pi_1(X_K)] 
\longrightarrow
\Z
\longrightarrow
0.
\end{equation}
Here, for $i,j \in \{ 1, \ldots ,g \}$, the boundary maps $\partial_k$ are defined by 
\begin{equation*}
	\begin{aligned}
		\left( \text{the $(1,j)$-entry of $\partial_1$} \right) &= 1-x_j, \\
		\left(\text{the $(i,j)$-entry of $\partial_2$} \right) &= \left[ \frac{\partial \gamma_j}{\partial x_i } \right], \\[5pt]
		\left(\text{the $(i,1)$-entry of $\partial_3$} \right) &= \left[\psi \left(\frac{\partial \sigma }{\partial \widetilde{\gamma}_i } \right)\right],
	\end{aligned}
\end{equation*}
where $\dfrac{\partial \,\cdot\, }{\partial x }$ denotes the Fox derivative with respect to $x$. 
See \cite[Sections~4 and 5]{Lyn} for a proof that \eqref{eq:cellofuniv} is indeed a chain complex.
We say that an identity $\sigma$ realizes $X_K$ if the chain complex \eqref{eq:cellofuniv} is chain equivalent to $C_*(\widetilde{X}_K)$.
According to \cite{Sie}, among identities, there exists an identity $\sigma$ realizing $X_K$. 
In addition, we call an identity that realizes a $3$-manifold a taut identity.
See Remark~\ref{rmk:taut} for the tautness.

Suppose that $\sigma \in R_g * F_g$ is a taut identity realizing $X_K$.
In this case, $\sigma$ can be written in $\sigma = \prod_{i=1}^{2g}\omega_i \widetilde{\gamma}_{j_i}^{\varepsilon_i} \omega_i^{-1}$.
With respect to the chain complex \eqref{eq:cellofuniv}, as in \cite[p.~474]{Tro}, we can explicitly describe a diagonal approximation map $D^\sharp$ as follows.
Let $\{ h_1^{(1)}, h_2^{(1)}, \ldots, h_{g}^{(1)} \}$ be the canonical basis for
$C_1 (\widetilde{X}_K) = \Z[\pi_1(X_K)]^g .$
Similarly, let $\{h_1^{(2)}, h_2^{(2)}, \ldots, h_{g}^{(2)}\}$ and $h_{1}^{(3)}$ be the canonical bases of $C_2 (\widetilde{X}_K) = \Z[\pi_1(X_K)]^g$ and $C_3 (\widetilde{X}_K) = \Z[\pi_1(X_K)] ,$ respectively.
Then,
\begin{equation*}
D^\sharp (h_{1}^{(3)}) = 
\sum_{i=1}^{2g} \varepsilon_i \left( \sum_{\ell=1}^{g} \left[\frac{\partial \omega_i}{\partial x_\ell}\right] h_\ell^{(1)} \otimes [\omega_i] \, h_{j_i}^{(2)}\right)
\in C_1 (\widetilde{X}_K) \otimes C_2 (\widetilde{X}_K).
\end{equation*}
Here $D^\sharp$ is a diagonal approximation as in \eqref{eq:diagonal}.

\begin{Remark}\label{rmk:taut}
The tautness encodes a topological procedure describing the boundary operator of the $3$-cell in $X_K$ by identifying the faces of a polyhedron in pairs, and it can be checked by a diagrammatic argument (see, for example, \cite[p.~126]{Sie}, \cite[p.~470]{Tro}, or \cite[Section~2]{Nos1}).
Since this diagrammatic method is complicated, it is often omitted.
In this paper we likewise omit a detailed explanation of this diagrammatic method; instead, following the method of \cite[p.~126]{Sie}, we give a diagrammatic verification that the identity \eqref{eq:tiden} is taut, as shown in Fig.~\ref{fig:taut}.
\end{Remark}

\subsection{A taut identity corresponding to $X_{T(m,n)}$}\label{sec:tidenT}
In this subsection, we construct a taut identity for the closed $3$-manifold $X_{T(m,n)}$.
From now on, let $m$ and $n$ be relatively prime integers greater than $1$.
Choose integers $r$ and $s$ such that
\begin{equation}\label{eq:mrns}
mr + ns = 1, \qquad -n < r < 0 < s < m.
\end{equation}

For this purpose, we begin by recalling the torus knots $T(m,n)$.
The torus knot $T(m,n)$ is defined by the embedding
\begin{equation*}
T(m,n) \coloneqq \{ (e^{2 \pi \sqrt{-1} m \theta}, e^{2 \pi \sqrt{-1} n \theta}) \mid \theta \in \R\} \subset S^1 \times S^1 \subset S^3 ,
\end{equation*}
where the torus $S^1 \times S^1 \subset S^3$ is the boundary of a tubular neighborhood of the unknot in $S^3$.
It is well known (see, for example, \cite{Lic}) that
$
\pi_1(S^3 \setminus T(m,n)) \cong \langle x,y \mid x^m y^{-n} \rangle .
$
With respect to this presentation, a meridian of $T(m,n)$ is represented by
$
\mu \coloneqq x^s y^r,
$
and a preferred longitude is represented by $\mu^{-mn} y^n$ (see \cite{Nos3}).
Hence, we obtain a presentation
\begin{equation}\label{eq:preXT}
\pi_1(X_{T(m,n)}) \cong \langle x,y \mid x^m y^{-n}, \mu^{-mn} y^n \rangle .
\end{equation}
By the same procedure as in \cite[Section~7]{BCP3}, we can construct a genus--$2$ Heegaard splitting of $X_{T(m,n)}$ associated with the presentation \eqref{eq:preXT}.

We now give an explicit description of the taut identity realizing $X_{T(m,n)}$.
As in Section~\ref{sec:fox}, we consider the two free groups $F_2$ and $R_2$ and the map $\psi \colon R_2 * F_2 \to F_2$ associated to the presentation \eqref{eq:preXT}. 
That is, $\psi(\widetilde{\gamma}_1)=\gamma_1=x^m y^{-n}$ and $\psi(\widetilde{\gamma}_2)=\gamma_2=\mu^{-mn} y^n$.
\begin{Theorem}\label{thm:tiden}
\begin{equation}\label{eq:tiden}
(x^s \widetilde{\gamma}_1^{-1} x^{-s}) \, \widetilde{\gamma}_1 \, (\mu^{mn} \widetilde{\gamma}_2 \mu^{-mn}) \, (\mu^{mn+1} \widetilde{\gamma}_2^{-1} \mu^{-mn-1})
\end{equation}
is a taut identity for the presentation \eqref{eq:preXT} of $\pi_1 (X_{T(m,n)})$.
Moreover, this taut identity realizes $X_{T(m,n)}$.
\end{Theorem}
\begin{proof}
A direct computation shows that the word \eqref{eq:tiden} is an identity, and
Fig.~\ref{fig:taut} gives a diagrammatic verification that it is taut.

It remains to prove that this taut identity realizes $X_{T(m,n)}$.
It is known that, for any nontrivial knot $K$, the closed $3$-manifold $X_K$ is aspherical \cite{Gab}.
Moreover, \cite{BCP3} shows that if a presentation of the fundamental group of an aspherical closed oriented $3$-manifold is obtained from a Heegaard splitting, then there is a unique nontrivial identity modulo conjugation and Peiffer identities.
In addition, \cite{Sie} shows that, for every closed $3$-manifold $X$ and every balanced presentation of $\pi_1 (X)$, there exists a taut identity for that presentation realizing $X$.
Combining these facts, we conclude that, up to conjugation and Peiffer identities, there is a unique nontrivial taut identity realizing $X_{T(m,n)}$ for the presentation \eqref{eq:preXT} of $\pi_1(X_{T(m,n)})$.
Hence, the taut identity in the theorem must realize $X_{T(m,n)}$.
\begin{figure}[tbp]
  \centering
  \includegraphics[width=0.7\textwidth]{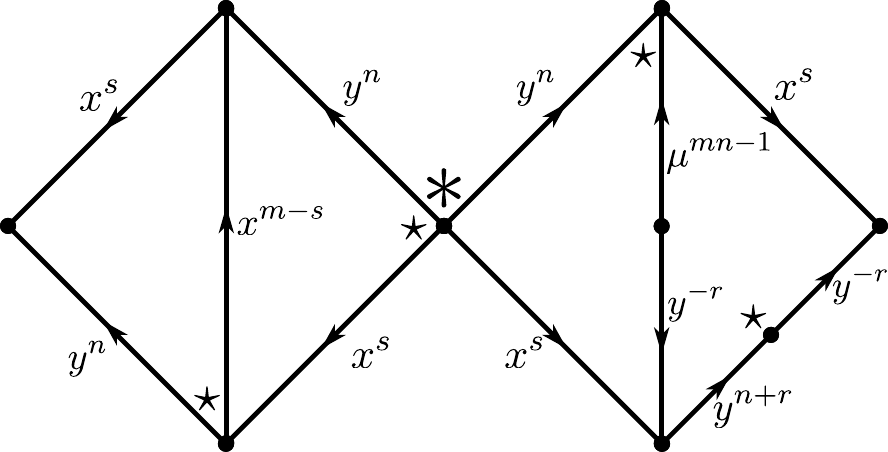}
  \caption{A net of the polyhedron corresponding to the taut identity \eqref{eq:tiden}.
  $*$ is the basepoint of the polyhedron, and $\star$ denotes the basepoint on each face.
  The two faces on the left are oriented clockwise, while the two faces on the right are oriented counterclockwise.
  From left to right, the faces correspond to $\gamma_1^{-1}, \gamma_1, \gamma_2^{-1}$, and $\gamma_{2}$. 
  For $i \in \{1,2\}$, in the polyhedron represented by this diagram, we obtain a CW–complex structure on $X_{T(m,n)}$ by identifying the face corresponding to $\gamma_i$ with the face corresponding to $\gamma_i^{-1}$ so that the basepoints $\star$ and the labels $x$ and $y$ match}
  \label{fig:taut}
\end{figure}
\end{proof}
\begin{Remark}
We briefly recall some known facts about torus knots.

The Levine--Tristram signatures of torus knots are computed explicitly in \cite{Kea,Mat,Lit}. 
In particular, Kearton \cite{Kea} describes the Blanchfield pairing of a torus knot over 
$\R[t^{\pm1}]$ in terms of the primary decomposition of 
$H_1(X_K;\R[t^{\pm1}])$, and Nosaka determines this pairing up to multiplication by a constant
by using quandle theory \cite{Nos2}. 
Twisted Blanchfield pairings associated with metabelian representations
of Casson--Gordon type for the torus knots $T(2,2\ell+1)$ are computed in \cite{BCP3}.
For Seifert matrices of torus knots, see \cite[Section~7]{Mur}, and for  constructions of Seifert surfaces of torus knots, see \cite{Nak}.
\end{Remark}

\section{Proof of Theorem \ref{thm:preBl}}\label{sec:prBl}
In this section, we prove Theorem~\ref{thm:preBl}.
First, combining the taut identity in Theorem~\ref{thm:tiden} with the procedure of Section~\ref{sec:fox}, we recover a chain complex associated with $X_{T(m,n)}$ (see Propositions~\ref{pro:cellX1} and~\ref{pro:cellX2}).
Then, using these chain complexes, we explicitly describe a generator of $H^1\bigl(X_{T(m,n)}; \Lambda_\Z/(\Delta_{T(m,n)})\bigr)$ (see Lemma~\ref{lem:ba1}).
Finally, using this generator, we compute the Blanchfield pairing of $T(m,n)$ directly from the definition.

In this section, the integers $m,n,r$, and $s$ are as in \eqref{eq:mrns} in Section~\ref{sec:tidenT}.
For $u \in \Lambda_\Z^\times$ and $\ell \in \Z \setminus \{0\}$, we adopt the convention
\[
\dfrac{1-u^\ell}{1-u} \coloneqq 
\left\{
\begin{aligned}
&\sum_{i=0}^{\ell-1} u^i, &&\text{if $\ell>0$,}\\
&-\sum_{i=1}^{-\ell} u^{-i}, &&\text{otherwise.}
\end{aligned}
\right.
\]
We begin by recovering the chain complex $C_*(X_{T(m,n)}; \Lambda_\Z)$.
Note that, for the presentation \eqref{eq:preXT} of $\pi_1(X_{T(m,n)})$, the abelianization map
$
\pi_1(X_{T(m,n)}) \longrightarrow \langle t \rangle
$
sends $x$ to $t^n$ and $y$ to $t^m$.

\begin{Proposition}\label{pro:cellX1}
The chain complex $C_*(X_{T(m,n)}; \Lambda_\Z)$ is chain homotopy equivalent to the chain complex
\begin{equation*}
0 
\longrightarrow \Lambda_\Z 
\xrightarrow{\, \partial_3\, } \Lambda_\Z^2 
\xrightarrow{\, \partial_2\, } \Lambda_\Z^2 
\xrightarrow{\, \partial_1\, } \Lambda_\Z 
\longrightarrow 0 ,
\end{equation*}
where $\partial_1$, $\partial_2$, and $\partial_3$ are presented by
\begin{equation*}
\begin{aligned}
&\partial_1 = 
\begin{pmatrix}
1-t^{-n} & 1-t^{-m} 
\end{pmatrix}, \\[10pt]
&\partial_2 = 
\begin{pmatrix}
 t^{\,n-mn}\,\dfrac{1 - t^{mn}}{1 - t^{n}} &
 -\,t^{\,mr+n}\,\dfrac{1 - t^{mn}}{1 - t^{n}} \dfrac{1 - t^{ns}}{1 - t} \\[15pt]
 -\,t^{\,m-mn}\,\dfrac{1 - t^{mn}}{1 - t^{m}} &
 t^{\,mr+m}\,\dfrac{1 - t^{mn}}{1 - t^{m}} \dfrac{1 - t^{ns}}{1 - t}
\end{pmatrix}, \\[10pt]
&\partial_3 = 
\begin{pmatrix}
1-t^{-ns} \\
t^{-mn} - t^{-mn-1} 
\end{pmatrix}.
\end{aligned}
\end{equation*}
Moreover, the cochain complex $C^*(X_{T(m,n)}; \Lambda_\Z)$ is chain homotopy equivalent to the chain complex
\begin{equation*}
0 
\longleftarrow \Lambda_\Z 
\xleftarrow{\, \partial_3^{\# \top}\, } \Lambda_\Z^2 
\xleftarrow{\, \partial_2^{\# \top}\, } \Lambda_\Z^2 
\xleftarrow{\, \partial_1^{\# \top}\, } \Lambda_\Z 
\longleftarrow 0 .
\end{equation*}
For these (co)chain complexes, the homomorphism
$C^2(X_{T(m,n)}; \Lambda_\Z) \longrightarrow C_1(X_{T(m,n)}; \Lambda_\Z)$
presented by the matrix
\begin{equation*}
\begin{pmatrix}
 -\,t^{\,n}\,\dfrac{1 - t^{ns}}{1 - t^{n}} &
 -\,t^{\,mn + mr + n}\,\dfrac{1 - t^{ns}}{1 - t^{n}} \\[6pt]
 0 &
 -\,t^{\,mn + m}\,\dfrac{1 - t^{mr}}{1 - t^{m}}
\end{pmatrix}
\end{equation*}
induces the Poincar\'e duality isomorphism.
\end{Proposition}
\begin{proof}
By applying the method of Section~\ref{sec:fox} to the taut identity in Theorem~\ref{thm:tiden}, a direct computation yields the above formulas.
\end{proof}
\noindent
From this proposition, a straightforward calculation shows
\begin{equation}\label{eq:alex}
\Delta_{T(m,n)} = t^{-(m-1)(n-1)/2}\,\frac{(1-t)(1-t^{mn})}{(1-t^m)(1-t^n)}.
\end{equation}
For the Alexander polynomial of torus knots, see also \cite{Lic}.
\begin{Proposition}\label{pro:cellX2}
Let $\partial_1$ and $\partial_3$ be the matrices in Proposition~\ref{pro:cellX1}.
Then the cochain complex $C^*(X_{T(m,n)}; \Lambda_\Z / (\Delta_{T(m,n)}))$ is chain homotopy equivalent to the chain complex
\begin{equation*}
0 
\longleftarrow \Lambda_\Z / (\Delta_{T(m,n)})
\xleftarrow{\, \partial_3^{\# \top} \,} (\Lambda_\Z / (\Delta_{T(m,n)}))^2 
\xleftarrow{\, 0 \, } (\Lambda_\Z / (\Delta_{T(m,n)}))^2 
\xleftarrow{\, \partial_1^{\# \top} \,} \Lambda_\Z / (\Delta_{T(m,n)})
\longleftarrow 0 .
\end{equation*}
\end{Proposition}
\noindent
In order to compute the matrix presentation of the Blanchfield pairing, we now determine a basis for
$
H^1\bigl(X_{T(m,n)}; \Lambda_\Z/(\Delta_{T(m,n)})\bigr)
$
as a $\Lambda_\Z/(\Delta_{T(m,n)})$-module.
From Proposition~\ref{pro:cellX2}, we obtain the following.
\begin{Lemma}\label{lem:ba1}
As a $\Lambda_\Z/(\Delta_{T(m,n)})$-module, the vectors
$
\begin{pmatrix}
1 \\
1 
\end{pmatrix}
$ and $
\begin{pmatrix}
1 - t^n \\
1 - t^m 
\end{pmatrix}
$
form a basis for $C^1(X_{T(m,n)}; \Lambda_\Z/(\Delta_{T(m,n)}))$.
In particular, the cohomology group
$
H^1\bigl(X_{T(m,n)}; \Lambda_\Z/(\Delta_{T(m,n)})\bigr)
$
is a cyclic $\Lambda_\Z/(\Delta_{T(m,n)})$-module generated by $\begin{pmatrix} 1 \\ 1 \end{pmatrix}$.
\end{Lemma}
\begin{proof}
For the first statement of the lemma, it suffices to show that the determinant of the matrix
$
\begin{pmatrix}
1 & 1 - t^n \\
1 & 1 - t^m 
\end{pmatrix},
$
namely $t^n - t^m$, is a unit in $\Lambda_\Z/(\Delta_{T(m,n)})$.
Since $m$ and $n$ are relatively prime, we have $\gcd(m-n,mn)=1$.
Hence
$
t^n - t^m = t^n(1 - t^{m-n}),
$
and the only common root of $1 - t^{m-n}$ and $1 - t^{mn}$ over $\C$ is $t=1$.
Using the expression \eqref{eq:alex} for $\Delta_{T(m,n)}$, we see that $t=1$ is not a root of $\Delta_{T(m,n)}$, and therefore $t^n - t^m$ and $\Delta_{T(m,n)}$ have no common root over $\C$.
Since both polynomials are monic, it follows that their greatest common divisor in $\Q[t^{\pm 1}]$ is $1$, and hence $t^n - t^m$ and $\Delta_{T(m,n)}$ are relatively prime in $\Lambda_\Z$.
Thus $t^n - t^m$ is a unit in $\Lambda_\Z/(\Delta_{T(m,n)})$.
This proves that 
$
\begin{pmatrix}
1\\
1
\end{pmatrix}
$ and $
\begin{pmatrix}
1 - t^n \\
1 - t^m 
\end{pmatrix}
$
form a basis for
$C^1(X_{T(m,n)}; \Lambda_\Z/(\Delta_{T(m,n)}))$ as a $\Lambda_\Z/(\Delta_{T(m,n)})$-module.

The second statement follows from the fact that the image of $\partial_1^{\# \top}$ is generated by
$
\begin{pmatrix}
1 - t^n \\
1 - t^m 
\end{pmatrix},
$
so 
$
H^1\bigl(X_{T(m,n)}; \Lambda_\Z/(\Delta_{T(m,n)})\bigr)
\cong
C^1\bigl(X_{T(m,n)}; \Lambda_\Z/(\Delta_{T(m,n)})\bigr) \big/ \mathrm{Im}(\partial_1^{\# \top})
$
is a cyclic $\Lambda_\Z/(\Delta_{T(m,n)})$-module generated by $\begin{pmatrix}1\\1\end{pmatrix}$.
\end{proof}
\noindent
Finally, we prove Theorem~\ref{thm:preBl}.
Recall that
\[
B(m,n) \coloneqq
t^{-(m+1)(n+1)/2} \,
\dfrac{(1 - t^{rm})(1 - t^{sn})(t^m - t^n)^2}{(1 - t^m)(1 - t^n)}
\in \Lambda_\Z .
\]
\begin{proof}[Proof of Theorem \ref{thm:preBl}]
Since $\Delta_{T(m,n)}(1) = \pm 1$ \cite{Lic}, the element $1 - t \in \Lambda_\Z / (\Delta_{T(m,n)})$ is a unit.
From Lemma~\ref{lem:ba1}, and since $1-t$ is a unit, we may choose
\[
\begin{pmatrix}
1-t \\
1-t 
\end{pmatrix}
\]
as a generator of $H^1\bigl(X_{T(m,n)}; \Lambda_\Z/(\Delta_{T(m,n)})\bigr)$ as a $\Lambda_\Z/ (\Delta_{T(m,n)})$-module.
Furthermore, as a generator of $H_1(X_{T(m,n)}; \Lambda_\Z)$, we may choose the image of
$\begin{pmatrix}
1-t \\
1-t 
\end{pmatrix}$
under $\mathrm{PD} \circ \beta$.
From Proposition~\ref{pro:cellX1}, with respect to these choices of bases, the Blanchfield pairing is presented by
\begin{equation*}
\frac{1}{ \Delta_{T(m,n)} }
\begin{pmatrix}
1-t \\
1-t 
\end{pmatrix}^{\# \top}
\begin{pmatrix}
 -\,t^{\,n}\,\dfrac{1 - t^{ns}}{1 - t^{n}} &
 -\,t^{\,mn + mr + n}\,\dfrac{1 - t^{ns}}{1 - t^{n}} \\[15pt]
 0 &
 -\,t^{\,mn + m}\,\dfrac{1 - t^{mr}}{1 - t^{m}}
\end{pmatrix}
\partial_2^{\# \top}
\begin{pmatrix}
1-t \\
1-t 
\end{pmatrix} ,
\end{equation*}
which simplifies to $t^{mn} B(m,n)$.
\end{proof}
\begin{Remark}
A direct computation shows that
$
(t^{mn} B(m,n))^\# = B(m,n),
$
and hence
\begin{equation*}
( t^{mn} B(m,n) )^\# - t^{mn} B(m,n)
= t^{-(m+n)} \Delta_{T(m,n)} \,
\dfrac{(1 - t^{rm})(1 - t^{sn})(t^m - t^n)^2}{1 - t}.
\end{equation*}
In particular, $( t^{mn} B(m,n) )^\# - t^{mn} B(m,n)$ is divisible by $\Delta_{T(m,n)}$ in $\Lambda_\Z$.
Therefore $t^{mn} B(m,n)$ in $\Lambda_\Z/(\Delta_{T(m,n)})$ is fixed by the involution $\#$, which shows that the Blanchfield pairing is a Hermitian sesquilinear form with respect to~$\#$.
\end{Remark}

\section{Proof of Theorem \ref{thm:pretAl} and \ref{thm:SptBL}}\label{sec:2THM}
In this section, we shall prove Theorems~\ref{thm:pretAl} and \ref{thm:SptBL}.
The proofs are given in Sections~\ref{sec:pretAl} and \ref{sec:SptBL}, respectively.

Throughout this section, the integers $m,n,r$ and $s$ are the same as \eqref{eq:mrns} in Section~\ref{sec:tidenT}.
For $a \in \Z_n \setminus \{0\}$, let $\mathcal{O}(a)$ denote the local ring of germs of holomorphic functions at the root of unity $e^{2\pi \sqrt{-1}\, a/n}$.
Let $I$ denote the $m\times m$ identity matrix, and let $0_{k,\ell}$ denote the $k\times \ell$ zero matrix.
For an invertible square matrix $A$ over $\Lambda_\C$ or $\mathcal{O}(a)$ and $\ell \in \Z \setminus \{0\}$, we use the notation
\[
\dfrac{I-A^\ell}{I-A} \coloneqq 
\left\{
\begin{aligned}
&\sum_{i=0}^{\ell-1} A^i, &&\text{if $\ell>0$,}\\
&-\sum_{i=1}^{-\ell} A^{-i}, &&\text{otherwise.}
\end{aligned}
\right.
\]
Note that this convention does not require $I-A$ to be invertible.

\subsection{Preparations for the proofs}\label{sec:pfp}
In this subsection, we prepare for the proofs of Theorems~\ref{thm:pretAl} and \ref{thm:SptBL}.
We first review the Casson--Gordon type metabelian representations for torus knots (see, e.g., \cite[Section~6.3]{MP} and \cite[Section~3]{CKP}).
We then verify that the representation used here satisfies conditions~(I), (II), and~(III) from Section~\ref{sec:twiBla} for any coprime integers $m,n \ge 2$.

We recall the Casson--Gordon type metabelian representation for $T(m,n)$ from \cite[Proposition~3.2]{CKP}.
This representation is commonly used in the definition of twisted Blanchfield pairings (see e.g. \cite{MP, CKP}).
Let $\bm{b}=(b_1,\dots,b_m)\in \Z_n^m$ be an $m$-tuple with $b_1+\cdots+b_m=0$, and assume $\bm{b}\neq 0$.
With respect to the presentation \eqref{eq:preXT} of $\pi_1(X_{T(m,n)})$, define
$
\rho(\bm{b}) \colon \pi_1(X_{T(m,n)}) \longrightarrow \mathrm{GL}_m(\Lambda_\C)
$
by
\[
\rho(\bm{b})(x)=
\begin{pmatrix}
0_{m-1,1} & I_{m-1} \\
t & 0_{1,m-1}
\end{pmatrix}^{\! n},
\qquad
\rho(\bm{b})(y)=
t \cdot \mathrm{diag}\!\bigl(e^{2\pi \sqrt{-1}\, b_1/n},\dots,e^{2\pi \sqrt{-1}\, b_m/n}\bigr).
\]
Here $I_{m-1}$ denotes the $(m-1)\times(m-1)$ identity matrix. Set
\[
X\coloneqq \rho(\bm{b})(x),\qquad
Y\coloneqq \rho(\bm{b})(y),\qquad
M\coloneqq \rho(\bm{b})(\mu).
\]
One checks that $M = X^s Y^r$ and $M^{mn} = Y^n = X^m = t^n I$.
By abuse of notation, we also write $\rho(\bm{b})$ for its composition with the inclusion $\mathrm{GL}_m(\Lambda_\C)\hookrightarrow \mathrm{GL}_m(\mathcal{O}(a))$.

We now show that, for any relatively prime positive integers $m$ and $n$, the representation $\rho(\bm{b})$ satisfies conditions (I), (II), and (III) of Section~\ref{sec:twiBla}.
First, condition (I) follows directly from the definition of $X$ and $Y$.
Next, we verify condition (II).
To this end, we describe the twisted chain complex of $X_{T(m,n)}$ from Theorem~\ref{thm:tiden}.
\begin{Proposition}\label{prop:twicell}
Let $R=\Lambda_\C$ or $\mathcal{O}(a)$. Then the chain complex
$C_*(X_{T(m,n)};R^m)$ twisted by $\rho(\bm{b})$ is chain homotopy equivalent to the chain complex
\[
0 \longrightarrow R^m
\xrightarrow{\,\partial_3\,} R^{m}\oplus R^{m}
\xrightarrow{\,\partial_2\,} R^{m}\oplus R^{m}
\xrightarrow{\,\partial_1\,} R^{m}
\longrightarrow 0,
\]
where the boundary maps $\partial_1$, $\partial_2$, and $\partial_3$ are presented by the block matrices
\[
\begin{aligned}
&\partial_1=\begin{pmatrix}
I - X^{-1} & I - Y^{-1}
\end{pmatrix}, \\
&\partial_2=\begin{pmatrix}
 t^{-n}\,\dfrac{I - X^{m}}{I - X}\,X &
 -\,X^{\,1-s}\,\dfrac{I - X^{s}}{I - X}\;\dfrac{I - M^{mn}}{I - M}\,M \\[12pt]
 -\,t^{-n}\,\dfrac{I - Y^{n}}{I - Y}\,Y &
 -\,Y^{\,1-r}\,\dfrac{I - Y^{r}}{I - Y}\,X^{-s}\,\dfrac{I - M^{mn}}{I - M}\,M \;+\; \dfrac{I - Y^{n}}{I - Y}\,Y
\end{pmatrix}, \\
&\partial_3=\begin{pmatrix}
I - X^{-s} \\[2pt]
t^{-n}\,(I - M^{-1})
\end{pmatrix}.
\end{aligned}
\]
Moreover, the cochain complex $C^*(X_{T(m,n)};R^m)$ twisted by $\rho(\bm{b})$ is chain homotopy equivalent to the chain complex
\begin{equation*}
0 
\xleftarrow{\quad} R^m
\xleftarrow{\,\,\, \partial_3^{\# \top} \,\,\, } R^{m} \oplus R^{m} 
\xleftarrow{\,\,\, \partial_2^{\# \top} \,\,\, } R^{m} \oplus R^{m} 
\xleftarrow{\,\,\, \partial_1^{\# \top} \,\,\, } R^{m} 
\xleftarrow{\,\,\, \quad} 0 .
\end{equation*}
For these presentations, the homomorphism $C^2(X_{T(m,n)}; R^m) \longrightarrow C_1(X_{T(m,n)}; R^m)$
presented by the matrix
\begin{equation*}
\begin{pmatrix}
 - X \dfrac{I - X^s}{I - X} & - t^{n} X \dfrac{I - X^{s}}{I - X} Y^r \\[15pt]
 0_{m,m} & - t^{n} \dfrac{I - Y^{r}}{I - Y} Y
\end{pmatrix}
\end{equation*}
induces the Poincar\'e duality isomorphism.
\end{Proposition}
\noindent
Recall that $Q(\Lambda_\C)$ denotes the field of fractions of $\Lambda_\C$.
\begin{Proposition}
For all $i\in\Z$ we have $H_i\!\left(X_{T(m,n)}; Q(\Lambda_\C)^m\right)=0$.
Equivalently, $\rho(\bm{b})$ satisfies condition~\textnormal{(II)} from Section~\ref{sec:twiBla}.
\end{Proposition}
\begin{proof}
The proposition follows from Lemma~\ref{lem:rank} below.
\end{proof}
\begin{Lemma}\label{lem:rank}
For the matrices in Proposition~\ref{prop:twicell}, we have
\[
\operatorname{rank}_{\Lambda_\C}(\partial_1)
=
\operatorname{rank}_{\Lambda_\C}(\partial_2)
=
\operatorname{rank}_{\Lambda_\C}(\partial_3)
= m.
\]
\end{Lemma}
\begin{proof}
Since $\Lambda_\C$ is an integral domain, for all $i \in \Z$, we have
$
\operatorname{rank}_{\Lambda_\C}(\partial_i)=\operatorname{rank}_{Q(\Lambda_\C)}(\partial_i).
$
Thus it suffices to compute the ranks over the field of fractions $Q(\Lambda_\C)$ of $\Lambda_\C$.

$\det(I-X^{-1})=1-t^{-n}$ and $\det(I-X^{-s})=1-t^{-ns}$, so both $I-X^{-1}$ and $I-X^{-s}$ are invertible over $Q(\Lambda_\C)$.
Therefore $\operatorname{rank}_{Q(\Lambda_\C)}(\partial_1)=\operatorname{rank}_{Q(\Lambda_\C)}(\partial_3)=m$.
From the chain complex in Proposition~\ref{prop:twicell}, we obtain $\mathrm{rank}_{Q(\Lambda_\C)} \partial_2 \leq 2m -  \operatorname{rank}_{Q(\Lambda_\C)}(\partial_3) = m.$
On the other hand, the $(2,1)$-block of $\partial_2$ is diagonal and every diagonal entry is nonzero. 
Hence
$
m \;\le\; \operatorname{rank}_{Q(\Lambda_\C)}(\partial_2),
$
and therefore $\operatorname{rank}_{Q(\Lambda_\C)}(\partial_2)=m$.
\end{proof}
\noindent
Finally, we verify condition \textnormal{(III)}.
In this setting, we have the following equalities of orders:
\begin{equation}\label{eq:otherord}
\begin{aligned}
\mathrm{Ord}_{\Lambda_\C} H^1(X_{T(m,n)}; \Lambda_\C^m)
&=
\mathrm{Ord}_{\Lambda_\C} H_2(X_{T(m,n)}; \Lambda_\C^m) \\
&=
\bigl( \mathrm{Ord}_{\Lambda_\C} H_0(X_{T(m,n)}; \Lambda_\C^m) \bigr)^\# \\
&= 1 .
\end{aligned}
\end{equation}
The first equality follows from Poincar\'e duality. The second equality is given by \cite[Prop.~3.7]{FV}.
The third equality is shown by \cite[Lemma~4.1]{BCP2}.
Hence $H^1\!\left(X_{T(m,n)}; \Lambda_\C^m\right)=0$, and therefore $\rho(\bm{b})$ satisfies condition~\textnormal{(III)}.

\subsection{Proof of Theorem \ref{thm:pretAl}}\label{sec:pretAl}
In this section, we prove Theorem~\ref{thm:pretAl}.
Our approach is to explicitly describe bases of $\mathrm{Ker} \partial_1$ and $\mathrm{Im} \partial_2$ in the complex of Proposition~\ref{prop:twicell} (see Propositions~\ref{prop:kernel} and \ref{prop:image}).
Throughout this subsection, unless explicitly stated otherwise, all modules and homomorphisms are over $\mathcal{O}(a)$.

We begin by defining the matrices $P$, $V$, and $W$, which we use to describe bases of $\ker(\partial_1)$ and $\operatorname{im}(\partial_2)$.
The matrix $P$ is the diagonal matrix whose $(i,i)$-entry is
\[
\begin{cases}
1, & \text{if } a \neq -b_i,\\
0, & \text{otherwise.}
\end{cases}
\]
We then define $V$ and $W$ by
\begin{equation*}
V \coloneqq (I - P Y^{-1})^{-1},
\qquad
W \coloneqq (I - (I-P)X^{-1})^{-1}.
\end{equation*}
Note that
$\det (I - (I-P)X^{-1}) = 1 \in \mathcal{O}(a)^{\times}$, and
$\det (I - P Y^{-1})$ is the product of $1 - e^{-2 \pi \sqrt{-1} b_i / n} t^{-1}$ over all $i \in \{1,2, \cdots, m\}$ with $a \neq -b_i$, and hence it lies in $\mathcal{O}(a)^{\times}$.
\begin{Lemma}\label{lem:propmat} 
\noindent
(i) The matrices $Y$, $P$, $\dfrac{I - Y^n}{I - Y}$, and $V$ commute.

\medskip
\noindent
(ii) $P$ is a projection matrix, i.e.
\[
P^2=P,\qquad (I-P)^2=I-P,\qquad P(I-P)=(I-P)P=0_{m,m}.
\]
Moreover, $\ker(I-P)=\operatorname{Im}(P)$.

\medskip
\noindent
(iii) The matrix $V$ satisfies the following two equalities:
\[
(I-Y^{-1})\,V \;=\; I-(I-P)Y^{-1}, \qquad (I-P)=(I-P)\,V .
\]

\medskip
\noindent
(iv)
We have $\dfrac{1 - t^n}{1 - e^{-2 \pi \sqrt{-1} a/n} t}
= \prod_{j \in \Z_n \setminus \{ a \}} (1 - e^{-2 \pi \sqrt{-1} j/n} t)
 \in \mathcal{O}(a)^\times$.
In particular,
\begin{equation*}
\dfrac{1 - e^{-2 \pi \sqrt{-1} a/n} t}{1 - t^n} \,
\dfrac{I - Y^{n}}{I - Y} \,
(I-P)
= (I-P).
\end{equation*}

\medskip
\noindent
(v) The matrix $W$ satisfies the following three equalities:
\[
(I-P)(I-X^{-1})\,W=(I-P),\qquad PW=P,\qquad
\frac{I-X^m}{I-X}\,X\,P \;=\; W\,P\,\frac{I-X^m}{I-X}\,X\,P .
\]
\end{Lemma}
\begin{proof}
\noindent
(i) Since $Y$ and $P$ are diagonal, all of $Y$, $P$, $\dfrac{I-Y^n}{I-Y}$, and $V$ are diagonal. 
Hence they commute.

\medskip
\noindent
(ii) This follows from a direct computation.

\medskip
\noindent
(iii) A straightforward calculation shows 
\[
\left\{
\begin{aligned}
&I - Y^{-1} = (I - (I-P)Y^{-1})(I - P Y^{-1}), \\
&I - P = (I-P)(I - P Y^{-1}).
\end{aligned}
\right.
\]
Right-multiplying these equalities by $V=(I-PY^{-1})^{-1}$ yields the desired equalities.

\medskip
\noindent
(iv) The matrix $\dfrac{I - Y^{n}}{I - Y} (I-P)$ is the diagonal matrix whose $(i,i)$-entry is $0$ if $a \neq -b_i$, and $\dfrac{1 - t^n}{1 - e^{-2 \pi \sqrt{-1} a/n} t}$ otherwise. Hence the equality follows.

\medskip
\noindent
(v) Since $W = (I - (I-P)X^{-1})^{-1}$ and $I - (I-P)X^{-1} = (I-P)(I - X^{-1}) + P$, we obtain
$(I-P)(I - X^{-1}) W + P W = I$.
Left-multiplying this equality by $I-P$ and by $P$ gives the first and second identities, respectively.
The third equality follows by left-multiplying by $W$ in the following equality
\[
(I - (I-P)X^{-1}) \, \dfrac{I - X^m}{I - X} X P
=
P \dfrac{I - X^m}{I - X} X P.
\]
\end{proof}
To compute $H_1\!\left(X_{T(m,n)}; \mathcal{O}(a)^m\right)$, we determine $\Ker(\partial_1)$ and $\Ima(\partial_2)$ in the chain complex of Proposition~\ref{prop:twicell}.
We begin with a detailed description of $\Ker(\partial_1)$.

\begin{Proposition}\label{prop:kernel}
A basis for $\Ker(\partial_1)$ is given by the columns of
\begin{equation*}
\begin{pmatrix}
 t^{-n} \dfrac{I - X^{m}}{I - X} X\\[15pt]
 -t^{-n} \dfrac{I - Y^{n}}{I - Y} Y
\end{pmatrix} (I-P)
+
\begin{pmatrix}
 W\\[6pt]
 -V (I - X^{-1}) W
\end{pmatrix} P.
\end{equation*}
\end{Proposition}
\begin{proof}
Let $E$ denote the invertible $\mathcal{O}(a)$-matrix
\[
E \coloneqq
\begin{pmatrix}
I & 0_{m,m} \\
- V (I - X^{-1}) & V
\end{pmatrix}.
\]
One checks that
\begin{equation*}
\begin{aligned}
\begin{pmatrix}
 t^{-n} \dfrac{I - X^{m}}{I - X} X\\[15pt]
 -t^{-n} \dfrac{I - Y^{n}}{I - Y} Y
\end{pmatrix} (I-P)
&=
E^{-1}
\begin{pmatrix}
 t^{-n} \dfrac{I - X^{m}}{I - X} X\\[15pt]
 -t^{-n} \dfrac{I - Y^{n}}{I - Y}
\end{pmatrix} (I-P), \\[10pt]
\begin{pmatrix}
 W\\[6pt]
 -V (I - X^{-1}) W
\end{pmatrix} P
&=
E^{-1}
\begin{pmatrix}
 W\\[6pt]
 0_{m,m}
\end{pmatrix} P .
\end{aligned}
\end{equation*}
Since $E$ is invertible, $\Ker (\partial_1) = E^{-1} \Ker (E \partial_1)$.
Therefore, the claim follows from Lemma~\ref{lem:Edel} below.
\end{proof}

\begin{Lemma}\label{lem:Edel}
A basis for $\Ker(\partial_1 E)$ is given by the columns of
\begin{equation}\label{eq:kerE}
\begin{pmatrix}
 t^{-n} \dfrac{I - X^{m}}{I - X} X\\[15pt]
 -t^{-n} \dfrac{I - Y^{n}}{I - Y}
\end{pmatrix} (I-P)
+
\begin{pmatrix}
 W\\[6pt]
 0_{m,m}
\end{pmatrix} P .
\end{equation}
\end{Lemma}
\noindent
To prove Lemma~\ref{lem:Edel}, we need the following lemma.

\begin{Lemma}\label{lem:propW}
$\mathrm{Ker} (I-P)(I-X^{-1})$ is generated by the columns of $W P$.
\end{Lemma}
\begin{proof}
By Lemma~\ref{lem:propmat}~(ii) and (v),
$
(I-P)(I-X^{-1})(WP)=(I-P)P=0_{m,m},
$
hence each column of $WP$ is in $\Ker\bigl((I-P)(I-X^{-1})\bigr)$.
Hence $\bm{v}\in \Ker\!\bigl((I-P)(I-X^{-1})\bigr)$ is a linear combination of the columns of $WP$.
Using Lemma~\ref{lem:propmat}~(v), we obtain
\[(I-P)W^{-1}\bm v = (I-P)(I-X^{-1})WW^{-1}\bm v = (I-P)(I-X^{-1})\bm v = 0_{m,1},\]
hence $W^{-1}\bm v\in \Ker(I-P)$.
By Lemma~\ref{lem:propmat}\,(ii), $\Ker(I-P)=\Ima(P)$, so $W^{-1}\bm{v}$ is a linear combination of the columns of $P$.
Applying $W$ to both sides shows that $\bm{v}$ is a linear combination of the columns of $WP$.
Therefore the columns of $WP$ generate $\Ker\!\bigl((I-P)(I-X^{-1})\bigr)$.
\end{proof}
\begin{proof}[Proof of Lemma \ref{lem:Edel}]
By Lemma~\ref{lem:propmat}, we have
\begin{equation*}\label{eq:leftright}
( \partial_1 E ) \left(
\begin{pmatrix} 
 t^{-n} \dfrac{I - X^{m}}{I - X} X\\[15pt]
 -t^{-n} \dfrac{I - Y^{n}}{I - Y}
\end{pmatrix} (I-P)
+
\begin{pmatrix}
 W\\[6pt]
 0_{m,m}
\end{pmatrix} P \right)
= 0_{m,m} .
\end{equation*}
Thus each column vector of \eqref{eq:kerE} lies in $\Ker(\partial_1 E)$.
We next show that $\Ker(\partial_1 E)$ is generated by the columns of \eqref{eq:kerE}.
Write $\bm{v}=\bm{v}_1\oplus \bm{v}_2\in \mathcal{O}(a)^m\oplus \mathcal{O}(a)^m$ for $\bm{v}\in \Ker(\partial_1 E)$, i.e.
\begin{equation*}\label{eq:ker-condition}
e^{-2 \pi \sqrt{-1} a/n} t^{-1} (I-P)(I-X^{-1}) \bm{v}_1
+
(I-P)(I-Y^{-1}) \bm{v}_2
+
P \bm{v}_2
= 0 .
\end{equation*}
Left-multiplying \eqref{eq:ker-condition} by $P$ and using Lemma~\ref{lem:propmat}\,(ii) yields $P\bm{v}_2=0$.
Define
\begin{equation}\label{eq:vv1}
\bm{v}' 
\coloneqq 
\bm{v} 
- 
\dfrac{1-e^{-2 \pi \sqrt{-1} a/n} t}{1-t^n}
\begin{pmatrix}
 \dfrac{I - X^{m}}{I - X} X\\[15pt]
 - \dfrac{I - Y^{n}}{I - Y}
\end{pmatrix} (I-P)\bm{v}_2 .
\end{equation}
By Lemma~\ref{lem:propmat} and $P\bm{v}_2=0$, we have the decomposition
$
\bm{v}' = \bm{v}'_1 \oplus 0_{m,1} \in \mathcal{O}(a)^m \oplus \mathcal{O}(a)^m.
$
Moreover, from $\bm{v} \in \Ker(\partial_1 E)$, we have $\bm{v}' \in \Ker(\partial_1 E)$.
Therefore,
\[
e^{-2 \pi \sqrt{-1} a/n} t^{-1} (I-P)(I-X^{-1}) \bm{v}'_1
=
(\partial_1 E) \bm{v}'
=
0_{m,1}.
\]
By Lemma~\ref{lem:propW}, $\bm{v}'_1$ is a linear combination of the columns of $WP$.
Consequently, by \eqref{eq:vv1}, $\bm{v}$ is a linear combination of the columns in \eqref{eq:kerE}, so the latter generate $\Ker(\partial_1 E)$.

Finally, since $E$ is invertible, Lemma~\ref{lem:rank} gives
$\operatorname{rank}(\partial_1 E)=\operatorname{rank}(\partial_1)=m$, hence
$\dim \Ker(\partial_1 E)=m$.
As \eqref{eq:kerE} has exactly $m$ columns, these columns form a basis for $\Ker(\partial_1 E)$.
\end{proof}

\noindent
Next, we describe a basis for $\Ima(\partial_2)$ explicitly.
\begin{Proposition}\label{prop:image}
A basis for $\Ima \partial_2$ is given by the columns of 
\begin{equation*}\label{eq:Imapar2}
\begin{pmatrix}
 t^{-n} \dfrac{I - X^{m}}{I - X} X\\[15pt]
 -t^{-n} \dfrac{I - Y^{n}}{I - Y} Y
\end{pmatrix}.
\end{equation*}
\end{Proposition}
\begin{proof}
Consider the matrix
\[
\begin{pmatrix}
I & I - X^{-s} \\
0_{m,m} & t^{-n} (I - M^{-1})
\end{pmatrix},
\]
which is obtained from $\partial_3$ by adding an extra column block $\begin{pmatrix}
I \\
0_{m,m}
\end{pmatrix}$.
This matrix is invertible over $\mathcal{O}(a)$, since its diagonal blocks $I$ and $t^{-n}(I-M^{-1})$ are invertible.
Hence its column vectors form a basis for $C_2\!\left(X_{T(m,n)};\mathcal{O}(a)^m\right)$.
On the other hand, by \eqref{eq:otherord} we have $\Ker(\partial_2)=\Ima(\partial_3)$, and $\Ima(\partial_3)$ is generated by the columns of $\partial_3$.
Therefore a basis for $\Ima(\partial_2)$ is given by
\begin{equation*}
\partial_2
\begin{pmatrix}
I \\
0_{m,m}
\end{pmatrix}
=
\begin{pmatrix}
 t^{-n} \dfrac{I - X^{m}}{I - X} X\\[15pt]
 -t^{-n} \dfrac{I - Y^{n}}{I - Y} Y
\end{pmatrix}
\end{equation*}
as claimed.
\end{proof}

Finally, we give the proof of Theorem~\ref{thm:pretAl}.
Define
\begin{equation}\label{eq:Omega}
\Theta_{\bm{b}}(a) \coloneqq (I-P) + P \left( t^{-n} \dfrac{I - X^{m}}{I - X} X \right) P .
\end{equation}
We briefly recall the definitions of $P$ and $X$.
For $a \in \Z_n \setminus \{ 0 \}$ and an $m$-tuple $\bm{b} = (b_1, b_2, \ldots, b_m) \in \Z_n^m$ with $b_1 + b_2 + \cdots + b_m = 0$, the matrix $P$ is the $m \times m$ diagonal with $(i,i)$-entry is defined to be $1$ if $b_i \neq -a$, and $0$ otherwise.
Moreover,
\[
X =
\begin{pmatrix}
0_{m-1,1} & I_{m-1} \\
t & 0_{1,m-1} 
\end{pmatrix}^{n},
\]
where $I_{m-1}$ is the $(m-1) \times (m-1)$ identity matrix.

\begin{proof}[Proof of Theorem \ref{thm:pretAl}]
Lemma~\ref{lem:propmat} implies the following equalities:
\begin{equation*}
\begin{aligned}
\begin{pmatrix}
 t^{-n} \dfrac{I - X^{m}}{I - X} X\\[15pt]
 -t^{-n} \dfrac{I - Y^{n}}{I - Y} Y
\end{pmatrix} (I-P)
&=
\begin{pmatrix}
 t^{-n} \dfrac{I - X^{m}}{I - X} X\\[15pt]
 -t^{-n} \dfrac{I - Y^{n}}{I - Y} Y
\end{pmatrix} (I-P)\, \Theta_{\bm{b}}(a), \\[10pt]
\begin{pmatrix}
 t^{-n} \dfrac{I - X^{m}}{I - X} X\\[15pt]
 -t^{-n} \dfrac{I - Y^{n}}{I - Y} Y
\end{pmatrix} P
&=
\begin{pmatrix}
 W\\[6pt]
 -V (I - X^{-1}) W
\end{pmatrix} P \, \Theta_{\bm{b}}(a) .
\end{aligned}
\end{equation*}
Thus, right-multiplying the basis for $\Ker (\partial_1)$ in Proposition~\ref{prop:kernel} by $\Theta_{\bm{b}}(a)$ yields the basis for $\Ima (\partial_2)$ in Proposition~\ref{prop:image}.
Hence,
$\Ker (\partial_1) / \Ima (\partial_2)
\cong
\mathcal{O}(a)^m \big/ \Theta_{\bm{b}}(a) \, \mathcal{O}(a)^m .$
This proves Theorem~\ref{thm:pretAl}.
\end{proof}

\subsection{Proof of Theorem \ref{thm:SptBL}}\label{sec:SptBL}
In this section we prove Theorem~\ref{thm:SptBL}.
The outline of the proof is as follows.
First, we describe generators for $H^2\!\left(X_{T(m,n)}; \mathcal{O}(a)^m\right)$ (see Proposition~\ref{prop:twiH2}).
Using these generators, we construct generators of
$H^1\!\left(X_{T(m,n)}; \bigl(\mathcal{O}(a)/(\Delta_{T(m,n)}^{\rho(\bm{b})})\bigr)^m\right)$
and
$H_1\!\left(X_{T(m,n)}; \mathcal{O}(a)^m\right)$
via the Bockstein map and the Poincar\'e duality map (see Lemmas~\ref{lem:ba2} and~\ref{lem:twi1}, respectively).
Finally, from these generators, we compute the matrix presentation of the twisted Blanchfield pairing directly from the definition.

Throughout this subsection we assume $a \in \Z_n \setminus \{0 , -b_1, -b_2 , \ldots, -b_m\}$.
Note that, under this assumption, the matrices $I-M$ and $I-Y$ are invertible over $\mathcal{O}(a)$.
Unless explicitly stated otherwise, all modules are over $\mathcal{O}(a)$.

We begin by describing generators of $H^2 (X_{T(m,n)}; \mathcal{O}(a)^m)$.
\begin{Proposition}\label{prop:twiH2}
\noindent
A basis for $\Ker (\partial_3^{\# \top})$ is given by the columns of 
\begin{equation}\label{eq:twiH21}
\begin{pmatrix}
I \\
- t^{-n} (I-M)^{-1} (I - X^s)
\end{pmatrix}.
\end{equation}
\end{Proposition}
\begin{proof}
A direct computation shows that
$
\partial_3^{\# \top} 
\begin{pmatrix}
I \\
- t^{-n} (I-M)^{-1} (I-X^s)
\end{pmatrix}
= 0_{2m,m} .
$
Thus the column vectors of \eqref{eq:twiH21} are in $\Ker(\partial_3^{\# \top})$.
We now show that $\Ker(\partial_3^{\# \top})$ is generated by these columns.
Let $\bm{v} = \bm{v}_1 \oplus \bm{v}_2 \in \mathcal{O}(a)^m \oplus \mathcal{O}(a)^m$ be an element of $\Ker (\partial_3^{\# \top})$.
Then
$
(I - X^s) \bm{v}_1 = - t^{n} (I-M) \bm{v}_2 .
$
Since $I-M$ is invertible over $\mathcal{O}(a)$, we obtain
$
\bm{v}
=
\begin{pmatrix}
I \\
- t^{-n} (I-M)^{-1} (I-X^s)
\end{pmatrix}
\bm{v}_1,
$
so $\bm{v}$ is a linear combination of the columns of \eqref{eq:twiH21}.
In particular, these columns generate $\Ker(\partial_3^{\# \top})$.
Finally, since $I$ is invertible, the columns of \eqref{eq:twiH21} are linearly independent.
This proves the proposition.
\end{proof}
\noindent
Next, we describe generators of
$H^1\!\bigl( X_{T(m,n)}; \bigl(\mathcal{O}(a)/(\Delta_{T(m,n)}^{\rho(\bm{b})})\bigr)^m \bigr)$ via the Bockstein homomorphism.
\begin{Lemma}\label{lem:ba2}
The columns of
\begin{equation*}
\begin{pmatrix}
t^{n(1-m)/2} \, \delta_m (t) \, \mathrm{adj} \!\left( \dfrac{I - X^m}{I - X} \right) \\
0_{m,m}
\end{pmatrix} .
\end{equation*}
give a basis for 
$\Ker ( \partial_2^{\# \top} \otimes \mathcal{O}(a)/(\Delta_{T(m,n)}^{\rho(\bm{b})})  )$ as an $\mathcal{O}(a)/(\Delta_{T(m,n)}^{\rho(\bm{b})})$-module.
Here, $\mathrm{adj}\,A$ denotes the adjugate of a matrix $A$, and
\begin{equation}\label{eq:delta}
\delta_m(t) =
\left\{
\begin{aligned}
& \dfrac{1}{t^{1/2} - t^{-1/2}}, && \text{if $m/2 \in \Z$,} \\
& 1, && \text{otherwise.}
\end{aligned}
\right.
\end{equation}
Moreover, with respect to this basis and the basis in Proposition~\ref{prop:twiH2}, 
the Bockstein homomorphism is presented by the identity matrix $I$.
\end{Lemma}
\begin{proof}
Note that $\det \Theta_{\bm{b}}(a)  = t^{\,n(1-m)} (1 - t^n)^{m-1} .$
Thus, from Theorem \ref{thm:pretAl}, we may take
\begin{equation}\label{eq:twiO}
\Delta_{T(m,n)}^{\rho(\bm{b})} = t^{\,n(1-m)/2} (1 - t^n)^{m-1} \, \delta_m(t) .
\end{equation}
Recall that the map $\partial_2^{\# \top} / \Delta_{T(m,n)}^{\rho(\bm{b})} : C^1(X_{T(m,n)}; (\mathcal{O}(a)/\Delta_{T(m,n)}^{\rho(\bm{b})} )^m) \to C^2(X_{T(m,n)}; \mathcal{O}(a)^m)$ induces the Bockstein homomorphism (see Remark~\ref{rmk:bock}).
We compute
\begin{equation*}
\dfrac{1}{\Delta_{T(m,n)}^{\rho(\bm{b})}} \,
\partial_2^{\# \top}
\begin{pmatrix}
t^{n(1-m)/2} \, \delta_m (t) \, \mathrm{adj} \!\left( \dfrac{I-X^m}{I-X} \right) \\
0_{m,m}
\end{pmatrix}
=
\begin{pmatrix}
I \\
- t^{-n} (I-M)^{-1} (I-X^s)
\end{pmatrix}
\end{equation*}
from Remark~\ref{rmk:adj} below.
Since the Bockstein homomorphism is an isomorphism, this proves the lemma.
\end{proof}

\begin{Remark}\label{rmk:adj}
We have
$
\dfrac{I - X^s}{I - X} \,
\mathrm{adj} \!\left( \dfrac{I - X^m}{I - X} \right)
= (1 - t^n)^{m-2} (I - X^s).
$
This equality follows from right-multiplying by $\mathrm{adj} \!\left( \dfrac{I - X^m}{I - X} \right)$ in the equality
\[
\dfrac{I - X^s}{I - X} (I - X^m)
=
(I - X^s) \dfrac{I - X^m}{I - X} .
\]
\end{Remark}
\begin{Remark}
In \cite{CKP}, it is shown that the twisted Alexander polynomial associated to the Casson--Gordon type metabelian representation $\rho(\bm{b})$ with coefficients in $\Lambda_\C$ is
\begin{equation}\label{eq:twi1}
\Delta_{T(m,n)}^{\rho(\bm{b})}
=
\dfrac{ (1 - t^n)^{m-1} }{ (1 - t) \prod_{i=1}^{m} (1-e^{2 \pi \sqrt{-1} b_i/n} t) } .
\end{equation}
If $a \in \Z_n \setminus \{0, -b_1, -b_2, \ldots, -b_m\}$, then over the local ring $\mathcal{O}(a)$ we have that each factor $1-e^{2 \pi \sqrt{-1} b_i/n} t$ and $1-t$ is a unit.
Hence \eqref{eq:twi1} agrees with \eqref{eq:twiO} up to multiplication by a unit in $\mathcal{O}(a)$.
\end{Remark}
\noindent
From the basis in Proposition~\ref{prop:twiH2} and the matrix presentation of the Poincar\'e duality map in Proposition~\ref{prop:twicell}, we obtain generators of $H_1\!\left(X_{T(m,n)}; \mathcal{O}(a)^m\right)$.
\begin{Lemma}\label{lem:twi1}
The columns of
\begin{equation*}
\begin{pmatrix}
X \dfrac{I - X^s}{I - X} Y^{r} (I-M)^{-1} (I - Y^{-r}) \\[15pt]
Y \dfrac{I - Y^{r}}{I - Y} (I-M)^{-1} (I - X^{s})
\end{pmatrix}
\end{equation*}
give a basis for $\Ker \partial_1$.
With respect to this basis and the basis in Proposition~\ref{prop:twiH2}, the Poincar\'e duality map is represented by the identity matrix $I$.
\end{Lemma}

We now prove Theorem \ref{thm:SptBL}.
Define
\begin{equation}\label{eq:Psi}
\Psi_{\bm{b}} (a) \coloneqq
t^{n(1-m)/2} \, (1-t^n)^{m-2} \, \delta_m(t) \,
(I - X^{-s})(I-M^{-1})^{-1}(I - Y^{-r}) .
\end{equation}
Recall the definitions of $X$, $Y$, and $M$:
\begin{equation*}
X =
\begin{pmatrix}
0_{m-1,1} & I_{m-1} \\
t & 0_{1,m-1} 
\end{pmatrix}^{n},
\quad 
Y = t \, \mathrm{diag} (e^{2 \pi \sqrt{-1} b_1/n}, e^{2 \pi \sqrt{-1} b_2/n}, \ldots, e^{2 \pi \sqrt{-1} b_m/n}),
\quad
M = X^{s} Y^{r} ,
\end{equation*}
where $I_{m-1}$ is the $(m-1) \times (m-1)$ identity matrix.

\begin{proof}[Proof of Theorem \ref{thm:SptBL}]
By definition, the matrix presentation of $\mathrm{Bl}^{\rho(\bm{b})}_{T(m,n)}$ with respect to the bases in Lemmas~\ref{lem:ba2} and~\ref{lem:twi1} is
\begin{equation*}
\begin{pmatrix}
t^{n(1-m)/2} \, \delta_m(t) \, \mathrm{adj} \!\left( \dfrac{I-X^m}{I-X} \right) \\
0_{m,m}
\end{pmatrix}^{\# \top}
\begin{pmatrix}
X \dfrac{I-X^s}{I-X} Y^{r} (I-M)^{-1} (I-Y^{-r}) \\[15pt]
Y \dfrac{I-Y^{r}}{I-Y} (I-M)^{-1} (I-X^{s})
\end{pmatrix} .
\end{equation*}
A direct computation shows that the product equals $t^{n} \Psi_{\bm{b}}(a)$.
This proves the theorem.
\end{proof}
\begin{Remark}
We have
$
(I - X^{-s})(I-M^{-1})^{-1}(I - Y^{-r})
=
-
(I - Y^{r})(I-M)^{-1}(I - X^{s})
$
which implies
$
\bigl( t^{n} \Psi_{\bm{b}} (a) \bigr)^{\# \top}
=
\Psi_{\bm{b}} (a) .
$
Hence,
\[
\bigl( t^{n} \Psi_{\bm{b}} (a) \bigr)^{\# \top}
-
\bigl( t^{n} \Psi_{\bm{b}} (a) \bigr)
=
t^{n(1-m)/2} \, (1-t^n)^{m-1} \, \delta_m(t) \,
(I - X^{-s})(I-M^{-1})^{-1}(I - Y^{-r}) .
\]
In particular, since $\Delta_{T(m,n)}^{\rho(\bm{b})}= t^{\,n(1-m)/2} (1 - t^n)^{m-1} \, \delta_m(t)$ over $\mathcal{O}(a)$, this shows that the twisted Blanchfield pairing is a Hermitian sesquilinear form with respect to the involution~$\#$.
\end{Remark}

\bibliographystyle{amsalpha}
\nocite{*}
\bibliography{main}

\end{document}